\documentclass[a4paper]{amsart}

\usepackage{amscd,amsthm,amsfonts,amssymb,amsmath}
\usepackage[colorlinks=true]{hyperref}
\usepackage{fullpage}
\usepackage[all]{xy}
\usepackage{mathrsfs}

\newcommand{\msl}{\mathscr{L}}

    \newcommand{\BA}{{\mathbb {A}}} 
    \newcommand{\BC}{{\mathbb {C}}} \newcommand{\BD}{{\mathbb {D}}}
     \newcommand{\BF}{{\mathbb {F}}}

    \newcommand{\BQ}{{\mathbb {Q}}}

     \newcommand{\BZ}{{\mathbb {Z}}}

    \newcommand{\CG}{{\mathcal {G}}} \newcommand{\CH}{{\mathcal {H}}}
     
    \newcommand{\CK}{{\mathcal {K}}} 
     \newcommand{\CN}{{\mathcal {N}}}
    \newcommand{\CO}{{\mathcal {O}}}

     \newcommand{\RN}{{\mathrm {N}}}

    \newcommand{\fa}{{\mathfrak{a}}} 
     
     \newcommand{\ff}{{\mathfrak{f}}}
    \newcommand{\fg}{{\mathfrak{g}}}

     \newcommand{\fp}{{\mathfrak{p}}}
    \newcommand{\fq}{{\mathfrak{q}}}

    \newcommand{\ad}{{\mathrm{ad}}}

    \newcommand{\Char}{{\mathrm{Char}}}

    \newcommand{\cyc}{{\mathrm{cyc}}}

    \newcommand{\Gal}{{\mathrm{Gal}}} 
    
    \newcommand{\Hom}{{\mathrm{Hom}}}

    \newcommand{\ord}{{\mathrm{ord}}} \newcommand{\rank}{{\mathrm{rank}}}
     \newcommand{\Pic}{\mathrm{Pic}}

    \renewcommand{\mod}{\ \mathrm{mod}\ }

    \newcommand{\Sel}{{\mathrm{Sel}}}

    \newcommand{\tor}{{\mathrm{tor}}}
    \newcommand{\RTr}{{\mathrm{Tr}}}

    \newcommand{\Vol}{{\mathrm{Vol}}}

    \font\cyr=wncyr10

    \newcommand{\Sha}{\hbox{\cyr X}}\newcommand{\wt}{\widetilde}
    \newcommand{\wh}{\widehat}

    \newcommand{\ov}{\overline}

    \newcommand{\sk}{\medskip}
    \newcommand{\lra}{\longrightarrow}
    \newcommand{\ra}{\rightarrow} 
    \newcommand{\bs}{\backslash}
    
    \newcommand{\s}{\sk\noindent}

    \theoremstyle{plain}
    \newtheorem{thm}{Theorem}[section] 
    \newtheorem{lem}[thm]{Lemma}  \newtheorem{prop}[thm]{Proposition}

\theoremstyle{remark} 
\theoremstyle{remark} 
\theoremstyle{remark} 

    \newcommand{\Neron}{N\'{e}ron~}

    \numberwithin{equation}{section}

\begin{document}

\title{On The Birch and Swinnerton-Dyer Conjecture for CM Elliptic Curves over $\BQ$}

\author{Yongxiong Li, Yu Liu and Ye Tian}

\dedicatory{\ \ \ \ \ \ \ \ \ \ \ \ \ \ \ \ \ \ \ \ \ \ \ \ \ \ \ \ \ \ \ \ \ \ \ \ \ \ \ \ \ \ \ \ \ \ \ \ \ \ \ \ \ \ \ \ \ \ \ \ \ \ \ \ \ \ \ \ \ \ \ \ \ \ \ \it To John Coates for his 70th birthday}

\maketitle

\tableofcontents

\section{Introduction and Main Theorems}
For an elliptic curve $E$ over a number field $F$, we write $L(s, E/F)$ for its complex $L$-function,  $E(F)$ for the Mordell-Weil group of $E$ over $F$, and $\Sha(E/F)$ for its Tate-Shafarevich group. For any prime $p$, let $\Sha(E/F)(p)$ or $\Sha(E/\BQ)[p^\infty]$ denote the $p$-primary part of $\Sha(E/F)$. When $F=\BQ$, we shall simply write $L(s,E)=L(s, E/\BQ).$
\begin{thm}\label{main thm} Let $E$ be an elliptic curve over $\BQ$ with complex multiplication. Let $p$ be any potentially good ordinary odd prime for $E$.
 \begin{enumerate}
 \item[(i)] Assume that  $L(s,E)$ has a simple zero at $s=1$. Then $E(\BQ)$ has rank one and $\Sha(E/\BQ)$ is finite. Moreover the order of $\Sha(E/\BQ)(p)$ is as predicted by the conjecture of  Birch and Swinnerton-Dyer conjecture.
\item[(ii)] If $E(\BQ)$ has rank one and $\Sha(E/\BQ)(p)$ is finite, then $L(E, s)$ has a simple zero at $s=1$.
\end{enumerate}
\end{thm}

\s{\em Remark}. The first part of (i) is the result of Gross-Zagier and Kolyvagin. The remaining part is due to Perrin-Riou for good ordinary primes. In this paper, we deal with odd bad primes which are potentially good ordinary. The result can be easily generalized to  abelian varieties over $\BQ$ corresponding to a CM modular form with trivial central character.

The following theorem shows that there are infinitely many elliptic curves over $\BQ$ of rank one for which the full BSD conjecture hold.
\begin{thm}\label{congruent} Let $n\equiv 5\mod 8$ be a squarefree positive integer, all of whose primes factors are congruent to $1$ modulo $4$. Assume that $\BQ(\sqrt{-n})$ has no ideal class of order $4$. Then the full BSD conjecture holds for the elliptic curve  $y^2=x^3-n^2 x$ over $\BQ$. In particular, for any prime $p\equiv 5\mod 8$, the full BSD holds for $y^2=x^3-p^2x$. 
\end{thm}
\begin{proof}By \cite{Tian} and \cite{TYZ}, the Heegner points constructed using the Gross-Prasad test vector is non-torsion. Thus both the analytic rank and Mordell-Weil rank of $E^{(n)}: y^2=x^3-n^2 x$ are one. By Perrion-Riou \cite{Perrin-Riou} and Kobayashi \cite{KO}, we know that the $p$-part of full BSD holds for all primes $p\nmid 2n$. Using the explicit Gross-Zagier formula in \cite{CST}, one can also check that the $2$-part of BSD holds by noting that $\dim_{\BF_2} \Sel_2(E^{(n)}/\BQ)/Im(E^{(n)}(\BQ)_\tor)=1$. By Theorem \ref{main thm}, the $p$-part of BSD also holds for all primes $p|n$,  since all primes $p$ with $p\equiv 1\mod 4$ are potentially good ordinary primes for $E^{(n)}.$
\end{proof}

Let $E$ be an elliptic curve defined over $\BQ$,  with complex multiplication (= CM in what follows) by an imaginary quadratic field $K$.  Let $p\neq 2$ be a potential good ordinary prime for $E$.  Note that $p$ must split in $K$,  and  also $p$ does not divide the number $w_K$ of roots of unity in $K$.

Assume that $L(s, E)$ has a simple zero at $s=1$. Choose an auxiliary imaginary quadratic field $\CK$ such that (i) $p$ is split over $\CK$ and (ii) $L(s, E/\CK)$ still has a simple zero at $s=1$. Let $E^{(\CK)}$ be the twist of $E$ over $\CK$, then $L(1, E^{(\CK)})\neq 0$. Let $\eta$ be the quadratic character associated to the extension $\CK/\BQ$ and $\eta_K$ its restriction to $K$. Let $\BQ_\infty$ be the cyclotomic $\BZ_p$-extension of $\BQ$,  and put $\Gamma=\Gal(\BQ_\infty/\BQ)$. For any finite order character $\nu$ of $\Gamma$, let $\nu_\CK$ denote its restriction to $\Gal(K\BQ_\infty/K)$. Consider the equality
$$L(s, E\otimes \nu) L(s, E^{(\CK)}\otimes \nu)=L(s, E_\CK\otimes \nu_\CK)$$
and its specialization to $s=1$.  Let $\msl_\varphi, \msl_{\varphi\eta_K}$ be the cyclotomic-line restrictions of the two Katz's two variable $p$-adic L-fucntion corresponding to $E$ and $E^{(\CK)}$, respectively.   Let $\msl_{E/\CK}$ be the cyclotomic-line restriction of the $p$-adic Rankin-Selberg L-function for $E$ over $\CK$.  The ingredients needed to prove the $p$-part BSD formula of $E$ are the following.
 \begin{enumerate}
 \item Rubin's two variable main conjecture\cite{rubin91} in order to relate the $p$-part of $\Sha(E/K)$ with $\msl_\varphi'(1)$. Note that $\ord_p(|\Sha(E/K)|)=2\ord_p(|\Sha(E/\BQ)|)$ for odd $p$.
 \item The complex Gross-Zagier formula \cite{YZZ} and the $p$-adic Gross-Zagier formula \cite{DD}, which relate
$\msl_{E/\CK}'(1)$ and $L'(1, E/\CK)=L'(1, E/\BQ)L(1, E^{(\CK)}/\BQ)$.
\item The precise relationship between $\msl_\varphi'(1)\msl{\varphi\eta_K}(1)$ and $\msl_{E/\CK}'(1)$,  and also between $\msl_{\varphi\eta_K}(1)$ and $L(1, E^{(\CK)})$.  This follows from the above equality of L-series and the interpolation properties of these $p$-adic L-fucntions.
\end{enumerate}
   Suppose that $E$ has bad reduction at $p$ which is potential good for $E$. Let $\fp$ denote a prime of $K$ above $p$. There is an elliptic curve $E'$ over $K$ with good reduction at $\fp$. In the process of proof, we need to compare periods, descends etc between $E$ and $E'$.

\s{\bf Notations}. Fix a non-trivial additive character $\psi: \BQ_p\lra \BC_p^\times$ with conductor $\BZ_p$.  For any character $\chi: \BQ_p^\times \ra \BC_p^\times$, say of conductor $p^n$ with $n\geq 0$,  we define the root number by
$$\tau(\chi, \psi)=p^{-n} \int_{v_p(t)=-n} \chi^{-1}(t)\psi(t) dt,$$
where $dt$ is the Haar measure on $\BQ_p$ such that $\Vol(\BZ_p, dt)=1$.
Fix embeddings $\iota_\infty: \ov{\BQ}\hookrightarrow \BC$ and $i_p: \ov{\BQ}\hookrightarrow \BC_p$ such that $\iota_p=\iota \circ \iota_\infty$ for an isomorphism $\iota: \BC\stackrel{\sim}{\ra} \BC_p$. For an elliptic curve $E$ over a number field $F$ and $p$ a potential good prime for $E$, let $(\ , )_\infty$ and $(\ ,\ )_p$ denote the normalized \Neron-Tate height pairing, and the $p$-adic height pairing with respect to cyclotomic character. Let $P_1, \cdots, P_r\in E(F)$ form a basis for $E(F)\otimes_\BZ \BQ$, define the regulars by
$$R_\infty(E/F)=\frac{\det\left((P_i, P_j)_\infty\right)_{r\times r}}{[E(K): \sum_i \BZ P_i]^2}, \qquad R_p(E/F)=\frac{\det\left((P_i, P_j)_p\right)_{r\times r}}{[E(K): \sum_i \BZ P_i]^2}.$$
For any character $\chi$ of $\wh{K}^\times$, let $\ff_\chi\subset \CO_K$ denote its conductor. For an elliptic curve $E$ over $K$, let $\ff_E$ denote its conductor.  For any non-zero integral ideals $\fg$ and $\fa$ of $K$, let $\fg^{(\fa)}$ denote the prime-to-$\fa$ part of $\fg$. Let $\BD$ be the completion of the maximal unramified extension of $\BZ_p$ and $\BD_\chi$ the finite extension of $\BD$ generated by the values of $\chi$. Let $L_\infty/K$ be an abelian extension whose Galois group $\CG=\Gal(L_\infty/K)\cong \Delta \times \Gamma$ with $\Delta$ finite and $\Gamma\cong \BZ_p^d$. Then for any $\BD[[\CG]]$-module $M$ and character $\chi$ of $\Delta$, put $M_\chi=M\otimes_{\BD[[\CG]], \chi} \BD_\chi[[\Gamma]]$. If $p\nmid |\Delta|$, let $M^\chi$ denote its $\chi$-component (as a direct summand).

\s{\em Acknowledgment.} The authors thank John Coates, Henri Darmon and Shouwu Zhang for their encouragement.

\section{Katz's $p$-adic L-function and Cyclotomic $p$-adic Formula}
Let $E$ be an elliptic curve defined over $K$ with CM by $K$ and $\varphi$ its associated Hecke character.  Let $p\nmid w_K$ be a prime split in $K$ and $p\CO_K=\fp\fp^*$ with $\fp$ induced by $\iota_p$. In particular, $K_\fp=\BQ_p$ in $\BC_p$ and let $\psi_\fp=\psi_p$ on $K_\fp$ under this identification.  Let $\Omega_E$ be a $\fp$-minimal period of $E$ over $K$. Let $\varphi$ be the associated Hecke character of $E$ and $\varphi_\fp$ its $\fp$-component. Let $\ff_E$ be the conductor of $\varphi$.

Let $F$ be an abelian extension over $K$ with Galois group $\Delta$. Assume that $p\nmid |\Delta|$ and denote by $\ff_{F/K}$ the conductor of $F$.  Let $\CG$ be the Galois group of the extension $F(E[p^\infty])$ over $K$. Then $\CG\cong \CG_\tor \times \Gamma_K$ with $\Gamma_K=\Gal(F(E[p^\infty])/F(E[p]))$. Let $\Lambda=\BZ_p[[\CG]]$.  Let $U_\infty$ and $C_\infty$ denote the $\BZ_p[[\CG]]$-modules formed from the principal local units at the primes above $\fp$, and the closure of the elliptic units for $K(E[p^\infty])$ (see \S4 of \cite{rubin91}  for the precise definitions.)

\begin{thm}[Two variable $p$-adic L-function] \label{measure2} Let $\fg$ be any prime-to-$p$ non-zero integral ideal of $K$. Assume that $\ff_E^{(p)}|\fg$.   There exists a unique measure $\mu_\fg=\mu_{\fg, \fp}$ on the group $\CG=\Gal(K(\fg p^\infty)/K)$ such that for any character $\rho$ of $\CG$ of type $(1, 0)$,
$$\rho(\mu_\fg)=\frac{\tau(\rho_\fp, \psi_\fp)}{\tau(\varphi_\fp, \psi_\fp)}\cdot \frac{1-\rho (\fp)p^{-1}}{1-\ov{\rho(\fp)}p^{-1}}\cdot \frac{L^{(\fg p)}(\ov{\rho}, 1)}{\Omega_E}.$$
Here $L^{(\fg p)}(\ov{\rho}, s)$ is the imprimitive L-series of $\ov{\rho}$ with Euler factors at places dividing $\ff p$-removed.
\end{thm}
\begin{proof}
It follows from the below lemma \ref{period} and  construction of Katz's two variable $p$-adic measure, see Theorem 4.14.
\end{proof}
\begin{thm}[Yager] For any character $\chi$ of $\CG_\tor$, let $\ff=\ff_\chi^{(p)}$ and $\mu_\ff^\chi:=\chi(\mu_\ff)\in \BD[[\Gamma_K]]$. Then we have
$$\Char(U_\infty/C_\infty)_\chi\cdot
\BD[[\Gamma_K]]=\left(\mu_{\ff}^\chi\right).$$
Here the measure $\mu_\ff$ is defined as in Theorem \ref{measure2}.
\end{thm}

\begin{lem}\label{period} Let $E/K$ be an elliptic curve associated with to a Hecke character $\varphi$,  $p$ splits in $K$ and write $p\CO_K=\fp\fp^*$.  Let $\varphi_0$ be a Hecke character over $K$ unramified at $\fp$.  Let $\Omega_E$ and $\Omega_0$ be  $\fp$-minimal periods of $E$ and $\varphi_0$, respectively.  Then
$$\ord_p\left(\frac{\Omega_E\cdot \tau(\varphi_\fp, \psi_p)}{\Omega_0}\right)=0.$$
\end{lem}
\begin{proof}
This follows from Stickelberger's theorem on prime ideal decomposition of Gauss sum. In fact, for $\fp\nmid w=w_K$, $E$ has $\fp$-minimal Weierstrass equation of form
$$E: y^2=x^3+a_2 x^2+a_4 x+a_6, \qquad a_2, a_4, a_6\in K^\times \cap \CO_\fp.$$
Note that for $w=4, 6$,  we may-and do- take form $y^2=x^3+a_4x$,  $y^2=x^3+a_6$, respectively.  Then there is an elliptic curve $E'$ over $K$ which has good reduction at $\fp$. Let $\varphi'$ be its associated Hecke character. Then $\epsilon=\varphi{\varphi'}^{-1}: \BA_K^\times/K^\times \ra \CO_K^\times$ (also viewed as a Galois character via class field theory) is of form $\chi(\sigma)=\sigma(d^{1/w})/d^{1/w}$ for an element $d\in K^\times/K^{\times w}$. Then the twist $E'$ has $\fp$-good model
$$E': \begin{cases} y^2=x^3+d a_2 x^2+d^2 a_4 x+ d^3 a_6, \qquad &\text{if $w=2$},\\
y^2=x^3+da_4 x, &\text{if $w=4$},\\
y^2=x^3+da_6, &\text{if $w=6$}.
\end{cases}$$
It is easy to check the $\Omega_{E_0}=d^{1/w}\cdot \Omega_E$. Let $\omega: \CO_\fp^\times \lra \mu_w\subset K$ be the character characterized by $\omega(a)\equiv a\mod \fp$ and let $\chi=\omega^{-(p-1)/w}$. Then $\epsilon_\fp=\chi^k$ for some $k\in \BZ/w\BZ$. Let $\kappa_\fp\cong \BF_p$ be the residue field of $K_\fp$. By Stickelberger's theorem, the Gauss sum $g(\epsilon_\fp, \psi):=-\sum_{a\in \kappa_\fp^\times} \epsilon_\fp(a) \psi(a)$ has $\fp$-valuation $\{k/w\}$. It remains to show that $k=\ord_\fp(d)$. Note that for any $u\in \CO_\fp^\times$, $K_\fp(u^{1/w})$ is unramified over $K_\fp$. Thus it is equivalent to show that for any uniformizer $\pi$ of $K_\fp$,
$$\sigma_u(\pi^{1/w})/\pi^{1/w}\equiv u^{-(p-1)/w}\mod \fp, \qquad \forall u\in \CO_\fp^\times.$$
But it is easy to see this by using local class field theory for formal group associated to $x^p-\pi x$.

For general Hecke character $\varphi_0$ over $K$ unramified at $\fp$ (not necessarily $K$-valued) and $\Omega_0$ its $\fp$-minimal period, it is easy to see that $\ord_\fp(\Omega_0/\Omega_{E_0})=0$.

\end{proof}
Let $\chi_{\cyc, K}: \CG \ra \BZ_p^\times$ be the $p$-adic cyclotomic character defined by the action on $p$-th power roots of unity.  Define
$$\msl_{\varphi_E}(s):=\mu_{\ff_E^{(\fp)}}(\varphi_E \chi_{\cyc, K}^{1-s}), \qquad \forall s\in \BZ_p.$$
Rubin's two variable main conjecture implies the following theorem.

\begin{thm}Let $E$ be an elliptic curve defined over $K$ with CM by $K$ and $\varphi$ its associated Hecke character.
Let $p\nmid w_K$ be a prime split in $K$ and $p\CO_K=\fp\fp^*$.  Let $r$ be the $\CO_K$-rank of $E(K)$. Assume that $\Sha(E/K)(p)$ is finite and the $p$-adic height pairing of $E$ over $K$ is non-degenerate. Then
\begin{enumerate}
\item both $\msl_\varphi(s)$ and $\msl_{\ov{\varphi}}(s)$ have a zero at $s=1$ of exact order $r$.
\item  the $p$-adic BSD conjecture holds for $E/K$:
$$\ord_p(|\Sha(E/K)|)=\ord_p\left(\frac{\msl_\varphi^{(r)}(1) \msl_{\ov{\varphi}}^{(r)}(1)}{R_p(E/K)}\cdot \prod_{v\mid p}\left((1-\varphi_E(v))\left(1-\ov{\varphi_E(v}\right)\right)^{-2}\right)$$
provided the assumption that if $w_K=4$ or $6$ then $E$ has bad reduction at both $\fp$ and $\fp^*$ or good reduction at both $\fp$ and $\fp^*$.
\end{enumerate}
Moreover, if $E$ is defined over $\BQ$, then we have
 $$\ord_p(|\Sha(E/\BQ)|)=\ord_p\left(\frac{\msl_\varphi^{(r)}(1) }{R_p(E/\BQ)}\cdot \prod_{v\mid p}\left((1-\varphi_E(v))\left(1-\ov{\varphi_E(v}\right)\right)^{-1}\right).$$
\end{thm}
\begin{proof}Let $\epsilon$ be a Galois character over $K$ valued in $\CO_K^\times$ such that $\varphi'=\varphi\epsilon$ is unramified at both $\fp$ and $\fp^*$.  Let $E'$ be the elliptic curve over $K$ as $\epsilon$-twist of $E$ so that $\varphi'$ as its Hecke character. Then $E'$ has good reduction above $p$. Let $F$ be the abelian extension over $K$ cut by $\epsilon$, then $[F:K]|w_K$. Moreover, $E$ and $E'$ are isomorphism over $F$, $E'(F)^{(\epsilon)}\cong E(K)$, and $\Sha(E'/F)[p^\infty]^{(\epsilon)}\cong \Sha(E/K)[p^\infty]$. Let $F_0=F(E[p])$ and $\chi: \Gal(F_0)/K)\ra \CO_\fp^\times$ be the character  giving the action on $E[p]$.

Let $F_\infty=F(E[p^\infty])$. Let $M_{\infty, \fp}$ be the maximal $p$-extension over $F_\infty$ unramified outside $\fp$ and $X_{\infty, \fp}=\Gal(M_{\infty, \fp}/F_\infty)$. Denote by $U_\infty$ and $C_\infty\subset U_\infty$ the $\Lambda=\BZ[[\Gal(F_\infty/K)]]$-modules of the principal local units at $\fp$ and elliptic units for the extension $F_\infty$ (defined as in \cite{rubin91}, \S 4). Rubin's two variable main conjecture, together Yager \cite{Yager}, says that
$$\Char_\Lambda(X^\chi_{\infty, \fp})\BD[[\Gal(F_\infty/F_0)]]=\left(\mu^\chi_{f_E^{(p)}, \fp}\right),$$
 where for an integral ideal $\fg$ of $K$ prime to $p$, the measure $\mu_\fg$ is given as in Theorem  \ref{measure2}. Let $\Sel(F_\infty, E[\fp^\infty])$ be the $\fp$-Selmer group of $E$ over $F_\infty$ and $\Sel(F_\infty, E[\fp^\infty])^\vee$ its Pontryajin dual. Then $\Sel(F_\infty, E[\fp^\infty])^\vee$ is a finitely generated $\Lambda$-torsion module and
 $$\Char_\Lambda (\Sel(F_\infty, E[\fp^\infty])^\vee)=\iota_\fp \Char (X_{\infty, \fp}^\chi),$$
 where $\iota_\fp: \Lambda\ra \Lambda, \gamma \lra \kappa_\fp(\gamma)\gamma$ for any $\gamma \in \Gal(F_\infty/K)$ and $\kappa_\fp$ is the character of $\Gal(F_\infty/K)$ giving the action on $E[\fp^\infty]$.   Similarly, we also have that $$\Char_\Lambda(X^\chi_{\infty, \fp^*})\BD[[\Gal(F_\infty/F_0)]]=\left(\mu^\chi_{f_E^{(p)}, \fp^*}\right),  \qquad \Char_\Lambda (\Sel(F_\infty, E[\fp^{* \infty}])^\vee)=\iota_{\fp^*} \Char (X_{\infty, \fp^*}^\chi).$$

 Let $F_\cyc$ be the cyclotomic $\BZ_p$ extension, and $\Lambda_\cyc=\BZ_p[[\Gal(F_\cyc/K)]]\cong \Delta \times \Gamma$ where $\Delta=\Gal(F/K)$ and $\Gamma=\Gal(F_\cyc/F)$.  Let $\Sel(F_\cyc, E[p^\infty])$ denote the $p$-Selmer group of $E$ over $F_\cyc$ and then its Pontryagin dual $\Sel(F_\cyc, E[p^\infty])^\vee$ is a finitely generated torsion $\Lambda_\cyc$-module. We have
 $$\begin{aligned}
 \Sel(F_\cyc, E[p^\infty])&=\Sel(F_\cyc, E[p^\infty])\oplus \Sel(F_\cyc, E[p^{*\infty}])\\
 &=\Hom(X_{\infty, \fp}, E[\fp^\infty])^{\Gal(F_\infty/F_\cyc)}\oplus \Hom(X_{\infty, \fp^*}, E[\fp^{*\infty}])^{\Gal(F_\infty/F_\cyc)}\\
 \end{aligned}$$
 Here the second equality is given \cite{PR1} Proposition (1.3), Theorem  (1.6) and Lemma (1.1), the last one is by the same reason as \cite{rubin91} Theorem 12.2. It follows that
 $$\Char_{\Lambda_\cyc}(\Sel(F_\cyc, E[p^\infty])^\vee)\BD[[\Gal(F_\cyc/F)]]=\left(\iota_\fp \mu^\chi_{f_E^{(p)}, \fp} \iota_{\fp^*} \mu^\chi_{f_E^{(p)}, \fp^*} \right).$$
 Denote by $\chi_\cyc$ the cyclotomic character. Let $f_E$ be a generater of $\Char_{\BZ_p[[\Gamma]]}\left(\Sel(F_\cyc, E[p^\infty])^{\vee}\right)^\Delta$ and define
 $$\msl(s)=\chi_\cyc^{1-s}(f_E), \qquad \forall s\in \BZ_p.$$
 Then we have $\msl(s)=u(s)\msl_{\varphi_E}(s)\msl_{\ov{\varphi_E}}(s)$ for some function $u(s)$ valued in $\BD^\times$.

Note that $E$ over $F$ has good reduction above $p$. Employing the descend argument as in \cite{Schneider1},  noting that the ``descent diagram" in \cite{Schneider1} \S 7 for $E$ over $F$ is $\Delta=\Gal(F/K)$-equivariant, and taking $\Delta$-invariant part, we have
\begin{prop} Let $r:=\rank_{\CO_K} E(K)$. Assume that $\Sha(E/K)[p^\infty]$ is finite and $p$-adic height pairing is non-degenerate on $E(K)$. Then $\msl(s)$ has exact vanishing order $2r$ at $s=1$ and if let $\msl^*(1)$ denote its leading coefficient at $s=1$,
$$\frac{\msl^*(1)}{R_p(E/K)}\sim  |\Sha(E/K)|\cdot \left|\prod_{v|p} H^1(\Gal(F(\mu_{p^\infty})/F), E(F(\mu_{p^\infty})\otimes_K K_v))^\Delta\right|^2.$$
Here for any $a, b\in \BC_p^\times$, write $a\sim b$ if $\ord_p(a/b)=0$.
\end{prop}
The follow lemma will complete the proof.
\begin{lem}Let $v_0=\fp$ or $\fp^*$. Assume that if $w_K=4$ or $6$ then  $E$ has bad reduction at both $\fp$ and $\fp^*$ or good reduction at both $\fp$ and $\fp^*$. Then
$$\left| H^1(\Gal(F(\mu_{p^\infty})/F), E(F(\mu_{p^\infty})\otimes_K K_{v_0}))^\Delta\right|\sim (1-\varphi_E(v_0))(1-\ov{\varphi_E(v_0)}).$$
\end{lem}
The remain part of this section will devote to the proof of this lemma. Note that \cite{Schneider1} handled the case where $E$ has good reduction above $p$. We now assume that $E$ has bad reduction either at $\fp$ or at $\fp^*$. The isomorphism between $E$ and $E'$ over $F$ gives rise to an isomorphism
$$H^1\left(\Gal(F(\mu_{p^\infty})/F), E(F(\mu_{p^\infty})\otimes_K K_{v_0})\right)^\Delta \stackrel{\sim}{\lra} H^1\left(\Gal(F(\mu_{p^\infty})/F), E'(F(\mu_{p^\infty})\otimes_K K_{v_0})\right)^\epsilon.$$ We will need  Proposition 2 in \cite{Schneider1} that for any elliptic curve $A$ over a local field $k$ with good ordinary reduction and let $\wt{A}$ denote its reduction over the the residue field $\kappa$ of $k$, we have
$$|H^1\left(\Gal(k(\mu_{p^\infty})/k), A(k(\mu_{p^\infty}))\right)=|\wt{A}(\kappa)[p^\infty]|.$$
Let $w|v_0$ be a place of $F$ above $v_0$ and $\kappa_w/\kappa_{v_0}$ be the residue fields of $F_w$ and $K_{v_0}$ respectively, we have
$$|E'(\kappa_w)|\sim \left(1-\varphi_{E'}(v_0)^{[\kappa_w: \kappa_{v_0}]}\right)\left(1-\ov{\varphi_{E'}(v_0)}^{[\kappa_w: \kappa_{v_0}]}\right).$$

If $W_K=2$, then $F/K$ is a quadratic extension. If $E$ is ramified at $v_0$, then $F/K$ is ramified at $v_0$ and let $w$ be the unique place of $F$ above $v_0$, we have $\kappa_w=\kappa_{v_0}$ and thus
$$\begin{aligned}
\left|H^1\left(\Gal(F(\mu_{p^\infty})/F), E'(F(\mu_{p^\infty})\otimes_K K_{v_0})\right)^\epsilon\right|&=
\frac{|H^1\left(\Gal(F_w(\mu_{p^\infty})/F_w), E'(F_w(\mu_{p^\infty}))\right)|}{|H^1\left(\Gal(K_{v_0}(\mu_{p^\infty})/K_{v_0}), E'(K_{v_0}(\mu_{p^\infty}))\right)|}\\
&=\frac{|\wt{E'}(\kappa_w)|}{|\wt{E'}(\kappa_{v_0})|}=1.\end{aligned}$$
If $E$ has good reduction at $v_0$, then $F/K$ is unramified at $v_0$. If $v_0$ is split over $F$, then $F\otimes_K K_{v_0}\cong K_{v_0}^2$ and $\epsilon_{v_0}=1$. It is easy to see
$$\left|H^1\left(\Gal(F(\mu_{p^\infty})/F), E'(F(\mu_{p^\infty})\otimes_K K_{v_0})\right)^\epsilon\right|\sim (1-\varphi_E(v_0))(1-\ov{\varphi_E(v_0)}).$$
If $v_0$ is inert in $F$, let $w$ be the unique prime of $F$ above $v_0$. Note that $\varphi'_{v_0}=\varphi_{v_0}\epsilon_{v_0}$ and $\epsilon(v_0)=-1$.
$$\begin{aligned}
\left|H^1\left(\Gal(F(\mu_{p^\infty})/F), E'(F(\mu_{p^\infty})\otimes_K K_{v_0})\right)^\epsilon\right|&=
\frac{|H^1\left(\Gal(F_w(\mu_{p^\infty})/F_w), E'(F_w(\mu_{p^\infty}))\right)|}{|H^1\left(\Gal(K_{v_0}(\mu_{p^\infty})/K_{v_0}), E'(K_{v_0}(\mu_{p^\infty}))\right)|}
=\frac{|\wt{E'}(\kappa_w)|}{|\wt{E'}(\kappa_{v_0})|}\\
&\sim \frac{(1-(\varphi\epsilon)(v_0)^2)(1-\ov{\varphi\epsilon(v_0)}^2)}
{(1-(\varphi\epsilon)(v_0))(1-\ov{\varphi\epsilon(v_0)})}
=(1-\varphi(v_0))(1-\ov{\varphi(v_0)})
\end{aligned}.$$

If $w_K=4$ or $6$, by our assumption,   $v_0$ must be ramified over $F$ and $\epsilon$ is non-trivial on its inertia subgroup. The proof is now similar to the previous ramified case.

\end{proof}

\section{$\infty$-adic and $p$-adic Gross-Zagier Formulae}

Let $E$ be an elliptic curve over $\BQ$ of conductor $N$ and $\phi$ its associated newform.  Let $p$ be a prime where $E$ is potential good ordinary or potential semi-stable. Let $\alpha: \BQ_p^\times \lra \BZ_p^\times$ be the character character contained in the representation $(V_pE)^{ss}$ of $G_{\BQ_p}$ such that $\alpha|_{\BZ_p^\times}$ is of finite order.

Let $\CK$ be an imaginary quadratic field such that $\epsilon(E/\CK)=-1$ and $p$ splits in $\CK$.  Let  $\Gamma_\CK$ be the Galois group of the $\BZ_p^2$-extension over $\CK$. Recall that \cite{DD} there exists a $p$-adic measure $\mu_{E/\CK}$ on $\Gamma_\CK$ such that for any finite order character $\chi$ of $\Gamma_\CK$
$$\chi(\mu_{E/\CK})=\frac{ L^{(p)}(1, \phi, \chi)}{8\pi^2(\phi, \phi)} \cdot \prod_{w|p}Z_w(\chi_w, \psi_w),$$
where $(\phi, \phi)$ is the Peterson norm of $\phi$:
 $$(\phi, \phi)=\iint_{\Gamma_0(N)\bs \CH} |\phi(z)|^2 dx dy, \qquad z=x+iy,$$
and for each prime $w|p$ of $\CK$, let $\alpha_w=\alpha\circ \RN_{\CK_w/\BQ_p}$ and $\psi_w=\psi_p \circ \RTr_{\CK_w/\BQ_p}$, and let $\varphi_w$ be a uniformizer of $\CK_w$, then
$$Z_w(\chi_w, \psi_w)=\begin{cases}(1-\alpha_w\chi_w(\varpi_w)^{-1})(1-\alpha_w \chi_w(\varpi_w)p^{-1})^{-1}, \qquad &\text{if $\alpha_w \chi_w$ is unramified}, \\
p^n \tau((\alpha_w\chi_w)^{-1}, \psi_w), &\text{if $\alpha_w \chi_w$ is of conductor $n\geq 1$}.
\end{cases}$$

The following lemma will be used to prove our main theorem.
\begin{lem}\label{local}  Let $E$ be an elliptic curve over $\BQ$ with CM by an imaginary quadratic field $K$. Assume $p$ is also split in $K$ write $p\CO_K=\fp\fp^*$ with $\fp$ induced by $\iota_p$, i.e. identify $K_\fp$ with $\BQ_p$ and the non-trivial element $\tau\in \Gal(K/\BQ)$ induces an isomorphism on $\BA_K$ and thus  $\tau: K_{\fp^*}\stackrel{\sim}{\ra} K_\fp=\BQ_p$.  Let $\varphi$ be its associated Hecke character. Then we have
$\alpha= \varphi_{\fp^*}\circ \tau^{-1} $ and $(\alpha^{-1}\chi_{\cyc})(x)=\varphi_\fp (x) x^{-1}$ for any $x\in \BQ_p^\times.$ Moreover,  for any place $w|p$ of $\CK$, any finite order character $\nu: \wh{\BQ}^\times/\BQ^\times \wh{\BZ}^{\times (p)}\BZ_{p, \tor}^\times \ra \mu_{p^\infty}$ viewed as character on $\Gamma_K$ by compose with norm
 $$Z_w(\alpha_w \nu_w, \psi)=\tau(\varphi_\fp \nu_p^{-1}, \psi)\cdot \frac{1-(\varphi_\fp \nu^{-1})(p)p^{-1}}{1-\ov{(\varphi_\fp \nu_p^{-1})}(p)p^{-1}}.$$
 \end{lem}
 \begin{proof}The claim follws from the relations  $\varphi \ov{\varphi}=|\ |^{-1}_{\BA_K^{(\infty)}}$ and $\varphi^\tau=\ov{\varphi}$.

 \end{proof}
 \medskip

 Let $\chi_{\cyc, \CK}: \Gamma_\CK\ra \BZ_p^\times$ denote the $p$-adic cyclotomic character of $G_\CK$.  Let $\chi$ be an anticyclotomic character. Define $\msl_{E/\CK, \chi}$ to be the $p$-adic L-function
$$\msl_{E/\CK, \chi}(s)=\mu_{E/\CK}(\chi\chi_{\cyc, \CK}^{s-1}), \quad s\in \BZ_p.$$
For trivial $\chi$, we write $\msl_{E/\CK}$ for $\msl_{E/\CK, \chi}$.
\begin{thm}[See \cite{YZZ} and \cite{DD}] \label{comparison} Let $E$ be an elliptic curve over $\BQ$ and $\CK$ an imaginary quadratic field. Let $p$ be a potentially good ordinary prime for $E$ and split over $\CK$. Assume that $\epsilon(E/\CK)=-1$. Then
$$\frac{\msl_{E/\CK, \chi}'(1)}{R_p(E/\CK, \chi)} \cdot \frac{L_p(E/\CK, \chi, 1)}{ \prod_{w|p} Z_w(\chi_w, \psi_w)}=\frac{L'(E/\CK, \chi, 1)}{ R_\infty(E/\CK, \chi)\cdot 8\pi^2(\phi, \phi)}.$$
Here $L_p(E/\CK, \chi, 1)$ is the Euler factor at $p$. In particular, $\msl_{E/\CK}'(1)=0$ if and only if $L'(E/K, 1)=0$.
\end{thm}
\begin{proof}
  Let $B$ be an indefinite quaternion algebra over $\BQ$ ramified exactly at the places $v$ of $\BQ$ where $\epsilon_v(E/\CK,\chi)\eta_v(-1)=-1$. It is known that there exists a Shimura curve $X$ over $\BQ$ (with suitable level) and a non-constant morphism $f: X\ra E$ over $\BQ$ mapping a divisor in Hodge class to the identity of $E$ such that its corresponding Heegner cycle $P_\chi(f)$ is non-trivial if and only if $L'(1, \phi, \chi)\neq 0$ by Theorem 1.2 in \cite{YZZ}, and if and only if $\msl'_{E/\CK, \chi}(1)\neq 0$ by Theorem  B in \cite{DD}. Thus $L'(E/\CK,\chi,1)=0$ if and only if $\msl'_{E/\CK,\chi}(1)=0$.

Now assume that $L'(E/\CK,1)\neq0$. By an argument of Kolyvagin, we know that  $(E(K_\chi)\otimes\CO_\chi)^\chi$ is of $\CO_\chi$-rank one,    $$\dfrac{\wh{h}_\infty(P_\chi(f))}{R_\infty(E/\CK,\chi)}=
\dfrac{\wh{h}_p(P_\chi(f))}{R_p(E/\CK,\chi)}\in\ov{\BQ}^\times.$$
	
By \cite{YZZ} theorem 1.2,
	$$
	\frac{L'(E/\CK, \chi, 1)}{ R_\infty(E/\CK, \chi)\cdot 8\pi^2(\phi, \phi)}=\dfrac{\wh{h}_{NT}(P_\chi(f))}{R_\infty(E/\CK,\chi)}\dfrac{4L(1,\eta)}{\pi c_\CK}\dfrac{L(1,\pi,\ad)}{8\pi^3(\phi,\phi)}\alpha^{-1}(f,\chi)$$
	and by \cite{DD} theorem B (with our definition of $\msl_{E/\CK, \chi}$),
	$$\dfrac{\msl_{E/\CK, \chi}'(1)}{R_p(E/\CK,\chi)}=\dfrac{h_p(P_\chi(f))}{R_p(E/\CK,\chi)}\dfrac{4L(1,\eta)}{\pi c_\CK}\dfrac{\prod_{w|p}Z_w(\chi_w,\psi_w)}{L_p(E/\CK,\chi,1)}\dfrac{L(1,\pi,\ad)}{8\pi^3(\phi,\phi)}\alpha^{-1}(f,\chi),$$
where the $\alpha(f,\chi)\in\ov{\BQ}^\times$. The theorem follows.
\end{proof}
Now we give an explicit form of $p$-adic Gross-Zagier formula as an application. Let $c$ be the conductor of $\chi$.  Assume the following Heegner hypothesis holds:
\begin{enumerate}
\item $(c, N)=1$,  and no prime divisor $q$ of $N$ is  inert in $\CK$,  and  also $q$ must be split in $\CK$ if $q^2|N$.
\item  $\chi([\fq])\neq a_q$ for any prime $q|(N, D)$, where $\fq$ is the unique prime ideal of $\CO_\CK$ above $q$  and $[\fq]$  is its class in $\Pic(\CO_c)$.
\end{enumerate}
Let $X_0(N)$ be the modular curve over $\BQ$, whose $\BC$-points parametrize isogenies $E_1\ra E_2$ between elliptic curves over $\BC$ whose kernel  is cyclic of order $N$. By  the Heegner condition, there exists a proper ideal $\CN$ of $\CO_c$ such that $\CO_c/\CN \cong \BZ/N\BZ$. For any proper ideal $\fa$ of $\CO_c$, let $P_\fa\in X_0(N)$ be the point representing the isogeny $\BC/\fa \ra \BC/\fa\CN^{-1}$, which is defined over the ring class field $H_c$ over $\CK$ of conductor $c$, and only depends on the class of $\fa$ in $\Pic(\CO_c)$.  Let $J_0(N)$ be the Jacobian of $X_0(N)$.   Let $f: X_0(N)\ra E$ be a modular parametrization mapping the cusp $\infty$ at infinity to the identity $O\in E$.  Denote by $\deg f$  the degree of the morphism $f$. Define the Heegner divisor to be
$$P_\chi(f):=\sum_{[\fa] \in \Pic(\CO_c)} f(P_\fa)\otimes \chi([\fa])\in E(H_c)_{\ov{\BQ}}.$$
\begin{thm}  Let $E, \chi$ be as above satisfying the Heegner conditions (1) and (2). Then
$$L'(1, E, \chi)=2^{-\mu(N, D)} \cdot \frac{8\pi^2 (\phi, \phi)_{\Gamma_0(N)}}{u^2 \sqrt{|Dc^2|}} \cdot \frac{\wh{h}_\infty(P_\chi(f))}{\deg f},$$
where $\mu(N, D)$ is the number of prime factors of the greatest common divisor of $N$ and $D$, $u=[\CO_c^\times:\BZ^\times]$ is half of the number of roots of unity in $\CO_c$, and $\wh{h}_\infty$ is the \Neron-Tate height on $E$ over $\CK$.

 Moreover, let $p$ be a prime split in $\CK$ and assume that $E$ is potential ordinary at $p$ (i.e. either potential good ordinary or potential semistable), then we have
$$\msl_{E/\CK, \chi}'(1)= \frac{ \prod_{w|p} Z_w (\chi_w, \psi_w)}{L_p(E/\CK, \chi, 1)}\cdot\frac{2^{-\mu(N, D)}}{u^2 \sqrt{|Dc^2|}} \cdot \frac{\wh{h}_p(P_\chi(f))}{\deg f},$$
where $\wh{h}_p$ is the $p$-adic height on $E$ over $\CK$.
\end{thm}
\begin{proof}The explicit form of Gross-Zagier formula is proved in \cite{CST}. The explicit form of $p$-adic Gross-Zagier formula then follows from the relation in Theorem \ref{comparison}.
\end{proof}
\section{Proof of Main Theorem \ref{main thm}}
In this section, let $E$ be an elliptic curve over $\BQ$ with CM by $K$ and $\Omega_E$ the minimal real period of $E$ over $\BQ$.   Let $p\nmid w_K$ be a prime split both in $K$.
\begin{lem}\label{comparison} Let $\CK$ be an imaginary quadratic field where $p$ splits, $\eta$ the associated quadratic character, and $\eta_K$ its base change to $K$. Assume that $\epsilon(E/\CK)=-1$. Then there exists a $p$-adic unit $u$ such that
$$\msl_{E/\CK}=\frac{\tau(\varphi_\fp, \psi_\fp)^2\cdot\Omega_E^2}{8\pi^2(\phi, \phi)}\cdot \msl_\varphi \msl_{\varphi\eta_K}.$$
\end{lem}
\begin{proof}
	It's enough to show that for any finite order character $ \nu:\wh{\BQ}/\BQ^\times\wh{\BZ}^{\times(p)}\BZ^\times_{p, \tor} \to\BC^\times$, we have
$$\nu_\CK(\mu_{E/K})=\tau^2(\varphi_p,\psi_p)\dfrac{\Omega_E^2}{8\pi(\phi,\phi)}\mu_{\ff_0}(\varphi\nu^{-1}_K)\mu_{\ff_0}(\varphi\eta_K\nu^{-1}_K).$$
Here $\nu_\CK=\nu\circ \RN_{K/\BQ}$ and $\nu_K=\nu\circ \RN_{K/\BQ}$. By interpolation property, the left hand of the formula in the lemma is
$$\dfrac{L^{(p)}(1,\phi,\nu^{-1}_\CK)}{8\pi^2(\phi,\phi)}\prod_{w|p}Z_w(\alpha_w\nu_w,\psi_w),$$
Note that $\CK/\BQ$ splits at $p$ and then $\eta_p$ is trivial, the right hand side of the formula in the lemma is
	$$\dfrac{\tau(\varphi_\fp\nu_\fp^{-1}, \psi_\fp)^2}{\tau(\varphi_\fp, \psi)^2}\cdot \left(\dfrac{1-\varphi\nu^{-1} (\fp)p^{-1}}{1-\ov{\varphi\nu^{-1} (\fp)}p^{-1}}\right)^2\cdot \dfrac{L^{(p\ff_0)}(\ov{\varphi\nu^{-1}}, 1)}{\Omega} \cdot \dfrac{L^{(p\ff_0)}(\ov{\varphi\nu^{-1}\eta_K}, 1)}{\Omega}$$
Then the formula follows from lemma \ref{local}.
\end{proof}
We are ready to prove Theorem \ref{main thm}. Assume that $L(s, E/\BQ)$ has a simple zero at $s=1$ and that $p$ is a bad but potentially good ordinary prime for $E$. Let $\varphi$ be the Hecke character associated to $E$ and $\ff_0$ its the prime-to-$p$ conductor.  We may choose an imaginary quadratic field $\CK$ such that
\begin{itemize}
\item $L(s, E/\CK)$ also has a simple zero at $s=1$.
\item $p$ is splits in $\CK$.
\item the discriminant of $\CK$ is prime to $\ff_0$.
\end{itemize}
Note that related Euler factors are trivial in this case, we then have
\begin{itemize}
\item $\displaystyle{\msl_{\varphi\eta_K}(1)=\frac{L(1, E^{(\CK)})}{\Omega_{E/K}}}$,
\item $\displaystyle{\frac{\msl_{E/\CK}'(1)}{R_p(E/\CK)\tau(\varphi_\fp, \psi_\fp)^2}=\frac{L'(E/\CK, 1)}{ R_\infty(E/\CK) 8\pi^2(\phi, \phi)}}$.
\item $\displaystyle{\ord_p(|\Sha(E/\BQ)|)=
    \ord_p\left(\frac{\msl_\varphi'(1)}{R_p(E/\BQ)}\right)}$,
\item $\displaystyle{\ord_p\left(\frac{\msl_{E/K}'(1)}{\msl_\varphi'(1) \msl_{\varphi\eta_K}(1)}\right)=\ord_p \left(\frac{\tau(\varphi_\fp, \psi_\fp)^2\Omega_{E/K}^2}{8\pi^2(\phi, \phi)}\right)}$,
\item $\displaystyle{\ord_p\left(\frac{\Omega_{E/K}}{\Omega_E}\right)=\ord_p
    \left(\frac{R_p(E/\CK)}{R_p(E/\BQ)}\right)=0}$.
    \end{itemize}
It follows that
$$\ord_p(|\Sha(E/\BQ|))=\ord_p\left(\frac{L'(E/\BQ, 1)}{\Omega_E\cdot R_\infty(E/\BQ)}\right).$$
This proves Theorem \ref{main thm} (i). Assume that $E(\BQ)$ has rank one and $\Sha(E/\BQ)(p)$ is finite, or equivalently, $E(K)$ has $\CO_K$-rank one and $\Sha(E/K)$ is finite. By \cite{Bertrand}, the cyclotomic $p$-adic height pairing is non-degenerate. Thus both $\msl_{\varphi_E}$ and $\msl_{\ov{\varphi_E}}$ have exactly order $1$ at $s=1$, therefore $\msl_{E/\CK}$ has exactly order one at $s=1$. It follows from $p$-adic Gross-Zagier formula that the related Heegner point is non-trivial and therefore $L(E, s)$ has a simple zero at $s=1$. This completes the proof of Theorem \ref{main thm}.

\end{document}